\documentclass{amsart}
\usepackage{hyperref}
\hypersetup{nesting=true,debug=true,naturalnames=true}
\usepackage{graphicx,amssymb,upref}

\usepackage{enumitem}

\usepackage[inline]{showlabels}

\usepackage{amssymb,amsmath,amscd}
\usepackage{url}
\usepackage{graphicx}
\usepackage{tikz}
\usepackage{pgfplots}
\usetikzlibrary{arrows,chains,matrix,positioning,scopes}

\newtheorem{theorem}{Theorem}[section]
\newtheorem{lemma}[theorem]{Lemma}
\newtheorem{proposition}[theorem]{Proposition}
\newtheorem{corollary}[theorem]{Corollary}
\newtheorem*{theorem*}{Theorem}
\newtheorem*{lemma*}{Lemma}
\newtheorem*{proposition*}{Proposition}
\newtheorem*{corollary*}{Corollary}

\theoremstyle{definition}
\newtheorem{remark}[theorem]{Remark}

\newtheorem*{remark*}{Remark}

\newcommand{\URA}[2]{\mathrm{rd}({#1})^{\wp}_{\scriptscriptstyle{#2}}}

\newcommand{\euO}{\mathfrak O}

\newcommand{\eus}{\mathfrak s}

\newcommand{\caO}{\mathcal O}
\newcommand{\caM}{\mathcal M}

\newcommand{\caL}{\mathcal L}

\newcommand{\bZ}{\mathbb Z}
\newcommand{\bF}{\mathbb F}

\newcommand{\id}{\mathrm{id}} 

\newcommand{\Tr}{\mathrm{Tr}} 
\newcommand{\Gal}{\mathrm{Gal}}
\newcommand{\chr}{\mathrm{char}}
\newcommand{\sep}{\mathrm{sep}} 
\newcommand{\df}{g\nu} 

\newcommand{\gQ}{Q_8}
\newcommand{\gD}{D_8} 
\newcommand{\gH}{H(p^3)}
\newcommand{\gM}{M(p^3)}

\title[Upper ramification sequences]{Upper ramification sequences of nonabelian extensions of degree $p^3$ in characteristic $p$}

\author{G.~Griffith Elder}

\address{Department of Mathematics \\
University of Nebraska at Omaha \\
Omaha, Nebraska, U.S.A.}
\email{elder@unomaha.edu}



\date{\today} 

\begin{document}

\begin{abstract}
 We classify the upper ramification breaks of totally ramified
 nonabelian extensions of degree $p^3$ over a local field of
 characteristic $p>0$. We find that nonintegral upper ramification
 breaks can occur for each nonabelian Galois group of order $p^3$, except the dihedral group of
 order $8$.
\end{abstract}

\maketitle

\section{Introduction}
Let $K$ be a local field and let $L/K$ be a finite Galois extension.
The Hasse-Arf theorem states that if $\Gal(L/K)$ is abelian, the
upper ramification breaks of $L/K$ are integers.  But what if $\Gal(L/K)$ is
nonabelian?  At least for small nonabelian extensions, the status of such basic
invariants should be well-understood.

In this paper, we give a complete description of the sequence of upper
ramification breaks for totally ramified nonabelian extensions of
degree $p^3$ over a local field of characteristic $p$.  We find that
the upper ramifications breaks are integers for dihedral extensions
but that for the other nonabelian groups there are extensions with
nonintegral upper ramification breaks.  In \cite{elder:hooper}, the
author and J. Hooper addressed this classification for quaternion
extensions over a local field of characteristic $0$ and residue
characteristic $2$ that contains the fourth roots of unity.

For abelian extensions, 
the sequence of ramification breaks are well understood due
to the work of several authors, including Miki and Thomas \cite{miki,thomas}.

\subsection{Nonabelian groups of order $p^3$}
\begin{equation}\label{gps}\arraycolsep=1.4pt\def\arraystretch{1.5}
\begin{array}{lrcl}
p=2:&\gQ&=&\langle\sigma_1,\sigma_2:\sigma_1^4=1,\sigma_1^2=\sigma_2^2=[\sigma_1,\sigma_2]\rangle,\\

&\gD&=&\langle\sigma_1,\sigma_2:\sigma_1^4=\sigma_2^2=1,\sigma_1^2=[\sigma_1,\sigma_2]\rangle,\\
p>2:&\gH&=&\langle\sigma_1,\sigma_2:\sigma_1^p=\sigma_2^p=[\sigma_1,\sigma_2]^p=1,[\sigma_1,\sigma_2]\in Z(\gH)\rangle,\\
&\gM&=&\langle\sigma_1,\sigma_2:\sigma_1^{p^2}=\sigma_2^p=1,
\sigma_1^p=[\sigma_1,\sigma_2]\rangle.
\end{array}\end{equation}
In all cases, the nonabelian group $G$ of order $p^3$ is
generated by two elements denoted by $\sigma_1,\sigma_2$, whose
commutator
$\sigma_3=[\sigma_1,\sigma_2]=\sigma_1^{-1}\sigma_2^{-1}\sigma_1\sigma_2$
generates the group's center $Z(G)$.
\begin{remark}\label{group-pres}
  Replacing $\sigma_1$ with $\sigma_1'=\sigma_1\sigma_2^{-a}$, $a\in \bZ$
  does not change any of these group presentations.
\end{remark}
The first two groups are recognizable as the quaternion and dihedral group, respectively. The third group is the Heisenberg group modulo $p$ and can be expressed in terms of matrices
\[\gH\cong\left\{\begin{bmatrix}
1&a&b\\
0&1&c\\
0&0&1\end{bmatrix}: a,b,c\in \bF_p\right\}\]
with entries in the field of $p$ elements. The fourth group can also be expressed in terms of matrices
\[\gM\cong\left\{\begin{pmatrix} a&b\\0&1\end{pmatrix}:a,b\in \bZ/(p^2), a\equiv 1\bmod p\right\}.\]
When $p=2$, note that the presentations of $\gM$ and $\gD$ agree.

\subsection{Local fields}

Throughout this paper,  $K$ is a field of characteristic $p$
that is complete with respect to a discrete valuation and has a perfect residue field.  Let $K^{\sep}$ be a separable closure of $K$, and for
each finite subextension $L/K$ of $K^{\sep}/K$ let $v_L$ be the
valuation on $K^{\sep}$ normalized so that $v_L(L^{\times})=\bZ$.  Let
$\caO_L$ denote the ring of integers of $L$, let $\caM_L$ denote the
maximal ideal of $\caO_L$, and let $\pi_L$ be a uniformizer for $L$.
Let $\bF_p$ be the field with $p$ elements.

\subsection{Ramification breaks}
We specialize the material in \cite[Chapter IV]{serre:local} to
our situation where $L/K$ is a totally ramified, Galois
extension of degree $p^3$. Define the lower ramification subgroups
$G_i\leq G=\Gal(L/K)$ by
\[G_i=\{\sigma\in \Gal(L/K): v_L(\sigma(\pi_L)-\pi_L)\geq i+1\}.\]
A lower ramification break occurs at $b$, if $G_b\supsetneq G_{b+1}$.
[Other authors may use ``jump,'' ``jump number'' or ``break number''
  rather than ``break''.]  Since the extension is totally ramified and
$G$ is a $p$-group, the lower ramification breaks are positive
integers coprime to $p$. If $[G_b:G_{b+1}]=p^m$, $m\geq 1$, we say
that $b$ occurs with multiplicity $m\geq 1$. Thus there are 3 breaks
$l_1\leq l_2\leq l_3$ in the lower ramification sequence.  The upper
ramification breaks $u_1\leq u_2\leq u_3$ are related to the lower
breaks by: $u_1=b_1$ and $u_i-u_{i-1}=(b_i-b_{i-1})/p^{i-1}$ for
$i\in\{2,3\}$.  The upper ramification groups $G^x$ for $0<x$ are
defined by setting $G^x=G_{b_1}=G$ for $x\leq u_1$, $G^x=G_{b_i}$ for
$u_{i-1}<x\leq u_i$ for $i=2,3$, and $G^x=\{\id\}$ for $u_3<x$.  When
passing from the ramification filtration of a Galois group $G$ to the
ramification filtration of a subgroup $H$, one uses the lower
ramification breaks; namely, $H_i=G_i\cap H$. The upper breaks are
used when passing to the ramification filtration of a quotient group
$G/H$; namely, $(G/H)^i=(G^iH)/H$.
\subsection{Special polynomials}
The Weierstrass $\wp$-function, $\wp(X)=X^p-X\in \bZ[X]$, is a $\bF_p$-linear map. Recall
from the theory of Witt vectors, the  Witt polynomial:
\[S(X_1,X_2)=\frac{X_1^p+X_2^p-(X_1+X_2)^p}{p}\in\bZ[X_1,X_2].\]

\subsection{Artin-Schreier extensions}\label{AS}
In characteristic $p$, cyclic extensions $L/K$ of
degree $p$ are Artin-Schreier. Thus $L=K(x)$
for some $x\in K^\sep$ such that $x^p-x=\kappa$ where $\kappa \in K$. We will refer to $\kappa$ as the
{\em
  Artin-Schreier generator} (AS-generator). As
explained in Remark \ref{technicalities}, we may replace $\kappa$ with
any $\kappa'\in \kappa+K^\wp$ where $K^\wp=\{\wp(k):k\in K\}$ without
changing the extension $L/K$. Thus, since $K$ is a local field, we may
assume, as we do in \S\ref{arith}, that the AS-generator $\kappa$ is
{\em reduced}; that is, $v_K(\kappa)=\max\{v_K(k):k\in\kappa+K^\wp\}$.
If $L/K$ is totally ramified, then $b:=-v_K(\kappa)=-v_L(x)>0$
is the ramification break of $L/K$, which is coprime to $p$.
For cyclic extensions of degree $p$ the upper and lower ramification breaks agree.

\subsection{Main Results}\label{results}
First, we describe the 
extensions in terms of its AS-generators. This result does not require $K$ to be a local field.

\begin{theorem}\label{as-gen}
  Let $K$ be a field of characteristic $p>0$.  An extension $L/K$ is a
  Galois extension with $\Gal(L/K)\cong \gQ, \gD,\gH$
  or $\gM$ if and only if $L=K(x_1,x_2,x_3)$ where
\[
  x_1^p-x_1=\kappa_1,\qquad
  x_2^p-x_2=\kappa_2,\qquad
  x_3^p-x_3=\eus(x_1,x_2)+\kappa_3,\]
  for some $\kappa_i\in K$ such that $\kappa_1,\kappa_2$
  represent $\bF_p$-linearly independent cosets of $K/K^\wp$, and
\[\eus(x_1,x_2)=-\kappa_2x_1+\begin{cases}
\kappa_1x_1+\kappa_2x_2 & \Gal(L/K)\cong \gQ,\\
\kappa_1x_1 & \Gal(L/K)\cong \gD,\\
0& \Gal(L/K)\cong \gH,\\
S(x_1,\kappa_1)& \Gal(L/K)\cong \gM.
\end{cases}\]
Note that $S(x_1,\kappa_1)=\kappa_1x_1$ for $p=2$.

Furthermore, the Galois group
$\Gal(L/K)=\langle\sigma_1,\sigma_2,\sigma_3\rangle$ is determined by
\[(\sigma_i-1)x_j=\delta_{ij},\]
the Kronecker delta function,
for all pairs $(i,j)$ with $1\leq i,j\leq 3$ except two:
$(i,j)=(1,3)$ and $(2,3)$.
For $(i,j)=(1,3)$, we have
\[(\sigma_1-1)x_3=-x_2+\begin{cases}
x_1 & \Gal(L/K)\cong \gQ,\\
x_1 & \Gal(L/K)\cong \gD,\\
0& \Gal(L/K)\cong \gH,\\
S(x_1,1)& \Gal(L/K)\cong \gM.
\end{cases}\]
For 
$(i,j)=(2,3)$, we have
\[(\sigma_2-1)x_3=\begin{cases}
x_2& \Gal(L/K)\cong \gQ,\\
0 & \Gal(L/K)\cong \gD,\gH,\gM.
\end{cases}\]
\end{theorem}
\begin{remark}\label{Witt/Saltman}
   That a theorem like Theorem \ref{as-gen} might exist is not a
   surprise. Saltman has proven that such descriptions exist for all
   Galois extensions with Galois group of order $p^n$, $n\geq 1$
   \cite[Corollary 2.5]{saltman}. What Theorem \ref{as-gen} does, that
   is not in \cite{saltman}, is explicitly describe the term 
   $\eus(x_1,x_2)$.
  \end{remark}

\begin{remark}\label{relations}
 Observe that for $p=2$, $S(x_1,\kappa_1)=\kappa_1x_1$ and
 $S(x_1,1)=x_1$.  Thus, unsurprisingly, the descriptions of
 $\eus(x_1,x_2)$ for $\gD$- and $\gM$-extensions agree.  Observe that
 for $p>2$, $\wp(x_i^2)=2x_i\kappa_i+\kappa_i^2$ and thus
 $x_i\kappa_i\in K(x_i)^\wp+K$ for $i=1,2$. This means that we can
 choose to set $\eus(x_1,x_2)=-\kappa_2x_1+\kappa_1x_1+\kappa_2x_2$
 for $\Gal(L/K)\cong \gH$, which would make the descriptions of
 $\eus(x_1,x_2)$ for $\gQ$- and $\gH$-extensions agree.  Finally,
 observe that $\wp(x_1x_2)=\kappa_2x_1+\kappa_1x_2+\kappa_1\kappa_2$.
 As a result, $-\kappa_2x_1\equiv
 \kappa_1x_2\pmod{K(x_1,x_2)^\wp+K}$. Replacing $x_3$ by $-x_3$, we
 can choose $\eus(x_1,x_2)=-\kappa_1x_2+\kappa_1x_1+\kappa_2x_2$ for
 $\Gal(L/K)\cong \gQ$ and choose $\eus(x_1,x_2)=-\kappa_1x_2$ for
 $\Gal(L/K)\cong \gH$.  These replacements show that the Galois
 extensions described in Theorem \ref{as-gen} for $\gQ, \gH$ remain
 invariant under the transposition $(1\;2)$ acting on subscripts, just
 as the descriptions of the two group presentations of groups are
 invariant. This is a comforting rather than a surprising observation.
\end{remark}
So that we may use Theorem \ref{as-gen} to determine ramification
breaks, we specialize the result record to local fields.

\begin{corollary}\label{ram-as-gen}
  Let $K$ be a local field of characteristic $p>0$.  An extension
  $L/K$ is a totally ramified Galois extension with $\Gal(L/K)\cong
  \gQ, \gD,\gH$ or $\gM$ if and only if the content of Theorem
  \ref{as-gen} holds, except that we replace the statement 
\begin{itemize}
\item  such that ``$\kappa_1,\kappa_2$
  represent $\bF_p$-linearly independent cosets of $K/K^\wp$''
\end{itemize}
with the alternate statement
\begin{itemize}
\item ``satisfying $v_K(\kappa_i)=-b_i<0$ with $p\nmid b_i$, and if
  $b_1=b_2$ then, without loss of generality, $\kappa_1,\kappa_2$
  represent $\bF_p$-linearly independent cosets in
  $\kappa_2\caO_K/\kappa_2\caM_K$,''
\end{itemize}
  With this replacement, the
  conclusions of Theorem \ref{as-gen} hold.
\end{corollary}

\begin{theorem} \label{sharp-bound}
  Let $K$ is a local field of characteristic $p>0$ with perfect
  residue field.  Let $M/K$ be a totally ramified $C_p^2$-extension
  with upper ramification breaks $u_1\leq u_2$. Therefore $M=K(y_1,y_2)$
  for some $y_1,y_2\in K^\sep$
  such that $\wp(y_i)=\beta_i$ with $v_K(\beta_i)=u_i$.
  Embed $M/K$ in a Galois extension $N/K$ with 
  Galois group $\Gal(N/K)=G\cong \gD,\gH,\gM$ or $\gQ$, as described in Theorem \ref{as-gen}.
  \begin{quote}
    {\rm [}Note: The $x_i$ in Theorem \ref{as-gen} are determined by the generators of the Galois groups, as described in \eqref{gps}. It is not necessarily the case that $y_i=x_i$. See Remark \ref{x-y-switch}.{\rm ]}
    \end{quote}
  For each group $G$, there is a lower bound $B_G$ such that the upper ramification breaks of
  $N/K$ are $u_1\leq u_2\leq u_3$ where $B_G\leq u_3$ such that if $B_G<u_3$ then $u_3$ is an integer coprime to $p$.
  In only remains to describe these lower bounds $B_G$:
  \[B_{\gD}=\begin{cases}
  u_1+u_2 &\mbox{for } u_2=u_1\mbox{ or }G_{l_2}=\langle \sigma_1^p,\sigma_2\rangle,\\
  2u_2 &\mbox{for } u_2\neq u_1\mbox{ and }G_{l_2}=\langle \sigma_1\rangle.
  \end{cases}
  \]
  By Lemma \ref{decomp}
  \[\beta_2=\mu_0^p+\sum_{i=1}^{p-1}\mu_i^p\beta_1^i\]
  for some $\mu_i\in K$ satisfying certain technical conditions stated there. Set
\[r=-v_K\left(\sum_{i=1}^{p-2}\mu_i^p\beta_1^i \right),\qquad s=-v_K(\mu_{p-1}^p\beta_1^{p-1}).\]
Observe
 $s\equiv -u_1\bmod p$, $r\not\equiv 0,-u_1 \bmod p$ and
$u_2=\max\{r,s\}$. 
Then
\[B_{\gH}=\max\left\{s+u_1,r+\frac{u_1}{p}\right\}.\]
If $\mu_{p-1}$, which is used to define $s$, satisfies $\mu_{p-1}\in -1+\caM_K$, set $\epsilon=\mu_{p-1}+1$ and
\[t=-v_K(\epsilon^p\beta_1^{p-1}).\]
Then
\[B_{\gM}=\begin{cases}
\max\left\{pu_1,s+u_1,r+\frac{u_1}{p}\right\} &\mbox{for }\mu_{p-1}\neq -1\bmod \caM_K,\mbox{ and }\\
&\hspace*{.75cm}u_2=u_1\mbox{ or }G_{l_2}=\langle \sigma_1^p,\sigma_2\rangle,\\
\max\left\{(p-1)u_1+\frac{u_1}{p},t+u_1,r+\frac{u_1}{p}\right\} 
&\mbox{for }\mu_{p-1}= -1\bmod \caM_K,\mbox{ and }\\
&\hspace*{.75cm}u_2=u_1\mbox{ or }G_{l_2}=\langle \sigma_1^p,\sigma_2\rangle,\\
pu_2&\mbox{for } u_2\neq u_1\mbox{ and }G_{l_2}=\langle \sigma_1\rangle.
\end{cases}\]

Using Lemma \ref{decomp}, $\beta_2\equiv \mu^p\beta_1\bmod \caO_K$. If $u_2=u_1$ then $\mu=\omega+\epsilon$ for some root of unity $\omega\in\caO_K/\caM_K$ and $\epsilon\in\caM_K$ satisfying $v_K(\epsilon)=e>0$.
Let
\[B_{\gQ}=\begin{cases}
2u_2 & \mbox{for } u_2\neq u_1,\mbox{ or }u_2= u_1,\omega^3\neq 1,\\
\max\left\{\frac{3u_1}{2},2u_1-2e\right\} & \mbox{for } u_2=u_1, \omega^3= 1.
\end{cases}
\]
\end{theorem}

\begin{remark}
  The conclusion of the Theorem of Hasse-Arf is the statement that the upper ramification breaks are integers. Because $B_{\gD}$ is an integer, the conclusion continues to hold for $G\cong \gD$. Because $B_G$ for $G\cong \gQ,\gH,\gM$ can fail to be an integer, the conclusion of Hasse-Arf can fail to hold for $G\cong \gQ,\gH,\gM$.
  \end{remark}

\subsection{Outline}
The link between reduced AS-generators and ramification breaks of
$C_p$-extensions, as described in \S\ref{AS}, is the tool we use to
determine the upper ramification breaks of our extensions.  Thus we
begin in \S\ref{embed} by deriving the AS-generators described in
Theorem \ref{as-gen} and Corollary \ref{ram-as-gen}.  Recall that
$l_1\leq l_2\leq l_3$ denote the lower ramification breaks of $N/K$.
Using \cite[Chapter IV, Proposition 10]{serre:local}, we will show that
$G_{l_3}=Z(G)=\langle\sigma_3\rangle$. Since $\langle\sigma_3\rangle$
is a ramification subgroup, the ramification break for $N/M$ is the
third lower ramification break $l_3$ of $N/K$. This means that to determine $l_3$, it is sufficient to reduce the AS-generator $\eus(x_1,x_2)+\kappa_3\in M$, except that the
notion of ``reduced,'' as given in \S\ref{AS},
is a little too simple for our purpose. Thus in
\S\ref{arith}, we generalize this notion
and apply it to determine the ramification breaks of certain auxiliary $C_p$-extensions. Then
in \S\ref{linking}, we use ramification theory to pull all these results together to determine
the ramification break of $M(x_3)/M$.
The upper ramification breaks $u_1\leq u_2\leq u_3$ follow.
We close in \S\ref{notting} by pointing out applications to the Nottingham group.

\section{Artin-Schreier generators}\label{embed} Let $K$ be a field of
characteristic $p>0$.  Notice: The results of this section do not require $K$ to be a local field.
Let $N/K$ be a Galois extension with
$\Gal(N/K)=\langle\sigma_1,\sigma_2\rangle \cong \gQ,
\gD,\gH$ or $\gM$, adopting the notation of \eqref{gps}.  Let $M=N^{\sigma_3}$ with
$\sigma_3=[\sigma_1,\sigma_2]$ denote the
fixed field of the center of $\Gal(N/K)$.  In every case,
$\langle\sigma_1,\sigma_2\rangle/\langle\sigma_3\rangle\cong
C_p^2$. Thus we may assume without loss of generality, $M=K(x_1,x_2)$ for
some $x_i\in N$ such that $\wp(x_i)=\kappa_i$ for some
$\kappa_i\in K$, $i=1,2$ that
represent $\bF_p$-linearly independent cosets of $K/K^\wp$, and
$(\sigma_i-1)x_j=\delta_{ij}$ for
$1\leq i,j\leq 2$.
\begin{remark}\label{technicalities}
  When we apply the results of this section in \S\ref{arith}, we will
  assume that in a preparatory step the AS-generators
  $\kappa_1,\kappa_2$ were adjusted in the two ways listed here. But
  since this preparatory step is not required in the remainder of this
  section, we do not yet assume that it has been done.
  
  \begin{enumerate}
  \item Observe $K(x_i)=K(x_i+\kappa)$ for all $\kappa\in
    K$. Thus $K(x_i)=K(x_i')$ for all $x_i'\in N$ satisfying
    $\wp(x_i')\in \kappa_i+K^\wp$ where $K^\wp=\{\wp(\kappa):\kappa\in
    K\}$. We may replace $\kappa_i$ with any element
    in $\kappa_i+K^\wp$ without changing 
    $M=K(x_1,x_2)$.
  \item Observe
  that $M=K(x_1,x_2)=K(x_1,x_2')$ for $x_2'=ax_1+ x_2$,
  $a\in\bZ$. We may replace $x_1$ with $x_1'$ while replacing
  $\kappa_1$ with $a\kappa_1+\kappa_2$ without changing the
  description of $M=K(x_1,x_2)$. Furthermore, if we replace
  $\sigma_1$ with $\sigma_1'=\sigma_1\sigma_2^{-a}$, then
  $\sigma_1',\sigma_2$ act on $x_1,x_2'$ as $\sigma_1,\sigma_2$ acted
  on $x_1,x_2$. By Remark \ref{group-pres}, replacing
  $\sigma_1$ with $\sigma_1'=\sigma_1\sigma_2^{-a}$ does not change
  the presentation of the Galois group.
  \end{enumerate}
\end{remark}

 Our description of the remaining part of the Galois extension, namely
$N/M$, depends upon the particular Galois group.
\subsection{$\gQ$}
Since $N/M$ is a quadratic extension, $N=M(x_3)$ for some $x_3\in N$ such that $\wp(x_3)\in M$ and $(\sigma_3-1)x_3=1$.
For $i=1,2$ there exist $A_i\in N$ such that $(\sigma_i-1)x_3=A_i$.
Since $[\sigma_i,\sigma_3]=1$ for $i=1,2$ we find that $A_i$ lies in the fixed field of $\sigma_3$. Thus $A_i\in M$.
Since $(\sigma_i-1)^2=(\sigma_3-1)$ for $i=1,2$ we find that
$(\sigma_i-1)A_i=1$. Thus
$A_i-x_i$ lies in the fixed field of $\langle\sigma_3,\sigma_i\rangle$. In other words, 
\[A_1-x_1\in K(x_2),\qquad A_2-x_2\in K(x_1).\]
There exist $a,b,c,d\in K$ such that $A_1=a+bx_2+x_1$ and $A_2=c+dx_1+x_2$.
Apply $\sigma_1\sigma_2=\sigma_2\sigma_1\sigma_3$ to $x_3$. The result is
$x_3+A_1+\sigma_1(A_2)=x_3+A_2+\sigma_2(A_1)+1$. Thus
\[(\sigma_1-1)A_2=(\sigma_2-1)A_1+1.\]
From this, we determine that $d=b+1$. Observe that $x_3'=x_3+dx_1x_2+ax_1+cx_2$
satisfies $\wp(x_3')\in M$, $(\sigma_3-1)x_3'=1$. Thus we may replace $x_3$ with $x_3'$ so that, without loss of generality,
\[(\sigma_1-1)x_3=x_1+x_2,\qquad (\sigma_2-1)x_3=x_2.\]

Set $T=\wp(x_3)\in M$. Since $(\sigma_1-1)x_3=x_1+x_2$, we find that $(\sigma_1-1)T=\kappa_1+\kappa_2$. Since $(\sigma_2-1)x_3=x_2$, $(\sigma_2-1)T=\kappa_2$.
Thus $T-(\kappa_1+\kappa_2)x_1-\kappa_2x_2$ is fixed by both $\sigma_1$ and $\sigma_2$. Thus $T=(\kappa_1+\kappa_2)x_1+\kappa_2x_2+\kappa_3$ for some $\kappa_3\in K$.

\subsection{$\gD, \gH, \gM$}
Let $L=N^{\sigma_2,\sigma_3}$ be
the fixed field of 
$\langle\sigma_2,\sigma_3\rangle$, which is a normal subgroup of order $p^2$. Thus $L=K(x_1)$ with $\wp(x_1)=\kappa_1$.
Observe that 
$N/L$ is a $C_p^2$ extension and consider
$N^{\sigma_2}/L$, a cyclic extension of degree $p$. The image of 
$\sigma_3$ generates $\Gal(N^{\sigma_2}/L)$.  So, without loss of generality,
$N^{\sigma_2}=L(x_3)$ for some  $x_3\in N$ such that $\wp(x_3)\in L$
and $(\sigma_3-1)x_3=1$. Since $x_3\in N^{\sigma_2}$, $(\sigma_2-1)x_3=0$.

It remains for us to determine $(\sigma_1-1)x_3=A\in N$.  Since
$[\sigma_1,\sigma_3]=1$, $A\in M$.  Apply
$\sigma_1\sigma_2=\sigma_2\sigma_1\sigma_3$ to $x_3$ to find that
$(\sigma_2-1)A=-1$. This means that $A+x_2$ is fixed by $\sigma_2$ and
thus $A+x_2\in L$. Further considerations depend upon the  group.

\subsubsection{$\gH$}
Apply the trace $\Tr_{L/K}$ to $A+x_2\in L$ and find 
$\Tr_{L/K}(A+x_2)=(\sigma_1^p-1)x_3+px_2=0$.  Thus, by the additive
version of Hilbert's Theorem 90, there is an $\ell\in L$ such that
$(\sigma_1-1)\ell=A+x_2$. Let $x_3'=x_3-\ell$. Observe that
$\wp(x_3')\in L$, $(\sigma_3-1)x_3'=1$, $(\sigma_2-1)x_3'=0$ and
$(\sigma_1-1)x_3'=-x_2$.  Thus without loss of generality, we relabel
so that $x_3$ has these properties:
\[(\sigma_2-1)x_3=0, \qquad (\sigma_1-1)x_3=-x_2.\]

Set $T=\wp(x_3)\in L$. Since $(\sigma_1-1)x_3=-x_2$, we find that
$(\sigma_1-1)T=-\kappa_2$. Thus $T+\kappa_2x_1\in L$ is fixed by
$\sigma_1$, which means $T+\kappa_2x_1\in K$ and
$T=-\kappa_2x_1+\kappa_3$ for some $\kappa_3\in K$.

\subsubsection{$\gD, \gM$}

Observe that $S(x_1,\kappa_1)=\kappa_1x_1$ when $p=2$. Thus $\gD$ is not a separate case.
Since it is well-known that
$S(x_1,1),S(x_1,\kappa_1)\in L$ satisfy $\Tr_{L/K}S(x_1,1)=1$ and
$(\sigma_1-1)S(x_1,\kappa_1)=\wp(S(x_1,1))$, we leave verification of these identifies to the reader; namely,
\[\Tr_{L/K}\left(\frac{x_1^p+1-(x_1+1)^p}{p}\right)=1,\]
\[(\sigma_1-1)\left(\frac{x_1^p+\kappa_1^p-(x_1+\kappa_1)^p}{p}\right)=\wp\left(\frac{x_1^p+1-(x_1+1)^p}{p}\right).\]

We are now ready to proceed. Apply the trace $\Tr_{L/K}$ to $A+x_2-S(x_1,1)\in L$ and find
$\Tr_{L/K}(A+x_2-S(x_1,1))=(\sigma_1^p-1)x_3+px_2-1=0$.  Thus, by the
additive version of Hilbert's Theorem 90, there is an $\ell\in L$ such
that $(\sigma_1-1)\ell=A+x_2-S(x_1,1)$. Let $x_3'=x_3-\ell$. Observe
that $\wp(x_3')\in L$, $(\sigma_3-1)x_3'=1$, $(\sigma_2-1)x_3'=0$ and
$(\sigma_1-1)x_3'=-x_2+S(x_1,1)$.  Thus without loss of generality, we
relabel so that $x_3$ has these properties:
\[(\sigma_2-1)x_3=0, \qquad (\sigma_1-1)x_3=-x_2+S(x_1,1).\]

Again, set $T=\wp(x_3)\in L$. Since $(\sigma_1-1)x_3=-x_2+S(x_1,1)$, we
find that $(\sigma_1-1)T=-\kappa_2+(\sigma_1-1)S(x_1,\kappa_1)$. Thus
$T+\kappa_2x_1-S(x_1,\kappa_1)\in L$ is fixed by $\sigma_1$. We conclude that
$T=-\kappa_2x_1+S(x_1,\kappa_1)+\kappa_3$ for some $\kappa_3\in
K$.

The converse follows from \cite[Corollary 2.5]{saltman}. However, since for these small extensions, one might want to see the details in the converse worked out, we provide a sketch. First, we introduce a lemma:

\begin{lemma}\label{converse}
  Let $K$ be a local field of characteristic $p>0$, $M/K$ be a $C_p^2$-extension with $\Gal(M/K)=\langle\bar{\sigma}_1,\bar{\sigma}_2\rangle$, and let
  $N/M$ be a $C_p$-extension with $N=M(x)$ for some $x\in K^\sep$
  such that $x^p-x=\mu$ with $\mu\in M$. Suppose that 
  \[(\bar{\sigma}_i-1)\mu\in M^\wp\]
  for $i=1,2$.  Then $N/K$ is Galois.
\end{lemma}
\begin{proof}
  By assumption, for $i=1,2$ there exist $\mu_i\in M$ such that
  $\bar{\sigma}_i(\mu)=\mu+\wp(\mu_i)$.  Both of
  $\bar{\sigma}_1,\bar{\sigma}_2$ can be extended to isomorphisms
  $\sigma_1,\sigma_2$ from $N$ into $K^\sep$. Observe that
  $\sigma_i(x)\in K^\sep$ is a root of $X^p-X=\bar{\sigma}_i(\mu)\in
  M[X]$. Thus $\sigma_i(x)-\mu_i$ is a root of $X^p-X=\mu$, which
  means that $\sigma_i(x)\in N$. The result follows.
  \end{proof}

Consider the case where $p>2$, $\wp(x_1)=\kappa_1$,
$\wp(x_2)=\kappa_2$, and
$\wp(x_3)=-\kappa_2x_1+S(x_1,\kappa_1)+\kappa_3$. To prove that $M(x_3)/K$ is Galois we use Lemma \ref{converse}. Apply $(\sigma_1-1)$ to 
the AS-generator
$-\kappa_2x_1+S(x_1,\kappa_1)+\kappa_3$ and the result is $-\kappa_2+(\sigma_1-1)S(x_1,\kappa_1)=\wp(-x_2+S(x_1,1))\in M^\wp$.
Apply $(\sigma_2-1)$ and the result is $0\in M^\wp$.

Now that $N/K$ is Galois, we identify  which Galois group by the relationships that $\sigma_1,\sigma_2$ satisfy on $x_3$.
Since
\[\wp((\sigma_2-1)x_3)=(\sigma_2-1)\wp(x_3)=0,\]
we determine that $(\sigma_2-1)x_3$ satisfies $X^p-X=0$ and thus $(\sigma_2-1)x_3=c\in\bF_p$. Therefore $(\sigma_2^p-1)x_3=0$ and $(\sigma_1-1)(\sigma_2-1)x_3=0$. We have proven $\lvert \sigma_2\rvert=p$.
Since
\[\wp((\sigma_1-1)x_3)=(\sigma_1-1)\wp(x_3)=-\kappa_2+\wp(S(x_1,1)),\]
we find that 
$(\sigma_1-1)x_3+x_2-S(x_1,1)=d\in\bF_p$. Thus
$(\sigma_1^p-1)x_3=(1+\sigma_1+\cdots \sigma_1^{p-1})S(x_1,1)=1$ and $(\sigma_2-1)(\sigma_1-1)x_3=-1$. We have proven $(\sigma_2^{p^2}-1)x_3=1$ and thus
$\lvert \sigma_2\rvert=p^2$.
Furthermore, putting together
\begin{align*}
(\sigma_1-1)(\sigma_2-1)x_3&=0,\\
  (\sigma_2-1)(\sigma_1-1)x_3&=-1,
  \end{align*}
we find that $\sigma_1\sigma_2(x_3)-\sigma_2\sigma_1(x_3)=1$, which means that $([\sigma_1,\sigma_2]-1)x_3=1$ and thus $[\sigma_1,\sigma_2]=\sigma_2^p$.
The other three cases are left for the reader.

\section{Reducing the Artin-Schreier generators}\label{arith}
Let $K$ be a local field of characteristic $p>0$ and let $N/K$ be a
totally ramified nonabelian extension of degree $p^3$ determined by
the Artin-Schreier equations given in Theorem \ref{as-gen} and
Corollary \ref{ram-as-gen}.  Thus the Galois group $G=\Gal(N/K)$ is
generated by $\sigma_1,\sigma_2$ where $\sigma_3=[\sigma_1,\sigma_2]$
generates the center $Z(G)$ of order $p$ and fixes a subfield
$M$, which is a
$C_p^2$-extension of $K$.  Recall the notation $G_i$ for the lower
ramification groups.  Let $u_1\leq u_2\leq u_3$ denote the upper and
$l_1\leq l_2\leq l_3$ denote the lower ramification breaks of $N/K$.

In \cite[Chapter IV, Proposition 10]{serre:local}, one sees that if
$\sigma\in G_i\setminus G_{i+1}$ and $\sigma'\in G_j\setminus
G_{j+1}$, then $[\sigma,\sigma']\in G_{i+j+1}$. As a result, the
elements of $G_{l_3}$ lie in the center $Z(G)$, and since $G_{l_3}$ is
nontrivial while the center $Z(G)$ has order $p$,
\[G_{l_3}=Z(G)=\langle \sigma_3\rangle.\]  Now that we have
proven that $M$ is the fixed field of a ramification group, we use
\cite[Chapter IV, Proposition 3, Corollary]{serre:local} to conclude
that the lower and upper ramification breaks of $M/K$ are $l_1\leq
l_2$ and $u_1\leq u_2$, respectively.  The ramification break of $N/M$
is $l_3$, and $u_3$ determined by $l_3-l_2=p^2(u_3-u_2)$.
As a result,  to determine the upper ramification sequence it only 
 remains to determinate the ramification break of the
$C_p$-extension $M(x_3)/M$ where \[\wp(x_3)=\eus(x_1,x_2)+\kappa_3\]
with $\kappa_3\in K$ and $\eus(x_1,x_2)$ described in Theorem \ref{as-gen}, and Corollary \ref{ram-as-gen}.
Thus the main object of this section is to 
``reduce'' this Artin-Schreier generator. 

Define the $K$-{\em
  group valuation}\footnote{This is the additive analog of the {\em defect} of
  a $1$-unit \cite[page 141]{wyman}.  }  of an element $\kappa\in K$ to be the
maximal valuation attained by the elements in the coset $\kappa+K^\wp$,
$K^\wp=\{\wp(\kappa):\kappa\in K\}$; namely,
\[\df_K(\kappa)=\max\{v_K(x):x\in\kappa+K^\wp\}.\]
Clearly, $\df_K$ is well-defined on the additive group $K/K^\wp$, and
$\df_K(\kappa)=\infty$ if and only if $\kappa\in K^\wp$,
while \[\df_K(\kappa_1+\kappa_2)\geq
\min\{\df_K(\kappa_1),\df_K(\kappa_2)\}\] with equality when
$\df_K(\kappa_1)\neq \df_K(\kappa_2)$.  Thus, once we compose with the
exponential function $\exp\circ\df_K: K/K^\wp\rightarrow
\Re_{>0}\cup\{\infty\}$, we have a function that satisfies the
conditions of a group valuation, a notion that Larson attributes to
Zassenhaus \cite{larson}. While there are four conditions
required of a group valuation, the remaining two hold vacuously since
addition is commutative and $\chr(K)=p$.

It is well-known that the $\bF_p$-linear map
\begin{equation}\label{wp}
  \wp: \caM_K^i/\caM_K^{i+1}\longrightarrow \begin{cases}
  \caM_K^i/\caM_K^{i+1}&\mbox{for }i\geq 0,\\
  \caM_K^{pi}/\caM_K^{pi+1}&\mbox{for }i< 0,\\
\end{cases}\end{equation}
is an isomorphism for $i\neq 0$, and for $i=0$ the kernel is
$\ker\wp=\bF_p$.  One consequence of this is that $\caM_K\subseteq
K^\wp$. Thus $\df_K(\kappa)=\infty$ for  $\kappa\in\caM_K$.
Another consequence is that for $\kappa\not\in \caM_K$,
$\df_K(\kappa)$ is either zero or equal to a negative integer coprime
to $p$.  This is used to prove that either $K(x)/K$ with
$\wp(x)=\kappa$ is unramified, or is ramified with $p\nmid
v_K(\kappa)<0$ and $b=-v_K(\kappa)$ the ramification break of
$K(x)/K$.

Recall that Remark \ref{technicalities} (1) states that
we may replace $\kappa_i$ with any $x\in \kappa_i+K^\wp$.
Thus we assume that this was done in \S\ref{embed}
so that each $\kappa_i$
is $K$-{\em reduced}\footnote{This is standard terminology. {\em e.g.} Reduced Witt vectors in \cite[\S4]{thomas}}; namely, $v_K(\kappa_i)=\df_K(\kappa_i)$.
 Since $N/K$ is totally ramified, the subextensions $K(x_i)/K$ are
ramified with ramification breaks $b_i=-v_K(\kappa_i)$.
Remark \ref{technicalities} (2) states that we may also
replace $x_2$ with $x_2'=ax_1+x_2$ and $\sigma_1$ with
$\sigma_1'=\sigma_1\sigma_2^{-a}$ without changing our description of
$M$ or the presentation of the group. We may then relabel so that not
only are $\kappa_1,\kappa_2$ reduced and $(\sigma_i-1)x_j=\delta_{ij}$
for $1\leq i,j\leq 2$, but if $v_K(\kappa_1)=v_K(\kappa_2)$,
equivalently $b_1=b_2$, then $\kappa_1,\kappa_2$ represent
$\bF_p$-linearly independent elements in
$\kappa_1\caO_K/\kappa_1\caM_K$.  Thus we may record that
\begin{equation}
  \label{one-break}
    u_1=\min\{b_1,b_2\}\mbox{ and }u_2=\max\{b_1,b_2\}
\end{equation}

Our current notation describes $M$ as $K(x_1,x_2)$ with subscripts
determined by the Galois group. In
\S\ref{embed}, the Galois group took center stage and this choice was natural.
The fixed field of $\langle\sigma_2,\sigma_3\rangle$
was $K(x_1)$ where $\wp(x_1)=\kappa_1$ and $v_K(\kappa_1)=-b_1$.  The
fixed field of $\langle\sigma_1,\sigma_3\rangle$ was $K(x_2)$ where
$\wp(x_2)=\kappa_2$ and $v_K(\kappa_2)=-b_2$.  In this subsection,
ramification takes center stage, which makes our notation inconvenient.
To address this,
we set $\{x_1,x_2\}=\{y_1,y_2\}$ such that
$\wp(y_i)=\beta_i\in K$ so that $K(y_1)/K$ and $K(y_2)/K$ have
ramification breaks $u_1=-v_K(\beta_1),u_2=-v_K(\beta_2)$, respectively.

\begin{remark} \label{x-y-switch}
Since the group presentations for $\gQ$ and $\gH$ are symmetric under
the transposition $(1\;2)$, we are able to assume that the subscripts
for the group presentations for $\gQ$ and $\gH$ were chosen from the
outset based upon the ramification filtration. Thus for these two
groups, $y_1=x_1$ and $y_2=x_2$. Only for the groups $\gD$ and $\gM$
does the introduction of $y_1,y_2$ matter. For these two groups,
because of \eqref{one-break}, we have $y_i=x_i$ when $b_1\leq b_2$,
and $y_1=x_2$, $y_2=x_1$ when $b_1>b_2$. Presenting this statement
another way, we have $y_i=x_i$ for $i=1,2$ when $u_2=u_1$ or
$G_{l_2}=\langle \sigma_1^p,\sigma_2\rangle$.  We have $y_1=x_2$,
$y_2=x_1$ when $u_2\neq u_1$ and $G_{l_2}=\langle \sigma_1\rangle$.
\end{remark}

Using Remark \ref{x-y-switch},
we translate
the formula for
$\eus(x_1,x_2)$ from 
Theorem \ref{as-gen} into expressions in $y_1,y_2,\beta_1,\beta_2$.
\begin{equation}\label{s(y,y)}
  \eus(x_1,x_2)=\begin{cases}
-\beta_2y_1+\beta_1y_1+\beta_2y_2 & \Gal(L/K)\cong \gQ,\\
-\beta_2y_1& \Gal(L/K)\cong \gH,\\
-\beta_2y_1+S(y_1,\beta_1)& \Gal(L/K)\cong \gD,\gM\mbox{ and } \\ &\hspace*{2cm} u_2=u_1\mbox{ or }G_{l_2}=\langle \sigma_1^p,\sigma_2\rangle,\\
-\beta_1y_2+S(y_2,\beta_2)& \Gal(L/K)\cong \gD,\gM\mbox{ and } \\ &\hspace*{2cm} u_2\neq u_1\mbox{ and }G_{l_2}=\langle \sigma_1\rangle.
  \end{cases}
   \end{equation}
Recall that for $p=2$, $\beta_iy_i=S(y_i,\beta_i)$ for $i=1,2$.
Furthermore, for
$\Gal(L/K)\cong \gD,\gM$. $u_2\neq u_1$ and $G_{l_2}=\langle \sigma_1\rangle$,
Remark \ref{relations} explains that
\[-\beta_1y_2=-\kappa_2x_1\equiv \kappa_1x_2=\beta_2y_1\pmod{M^\wp+K}.\]
Thus we record the following adjustment of \eqref{s(y,y)}: That
for $\Gal(L/K)\cong \gD,\gM$, $u_2\neq u_1$ and $G_{l_2}=\langle \sigma_1\rangle$, we may instead use
\begin{equation}\label{gD,gM-adjust}
\eus(x_1,x_2)=\beta_2y_1+S(y_2,\beta_2).
\end{equation}

Now we use the description of $\gH$-extensions
to motivate our next definition.
There we see that $N=M(x_3)$ where $x_3^p-x_3=-\beta_2y_1+\kappa_3$ for
some $\kappa_3\in K$. Furthermore, as $\kappa_3$ varies over all of $K$,
the field $N=M(x_3)$ varies over all $\gH$-extensions $N/K$
that contain $M=K(x_1,x_2)=K(y_1,y_2)$.
Any determination of the lower/upper ramification breaks of $N/K$, together with the ramification groups associated with them classifies the ramification breaks of
the $C_p$-extension $N/M$, and thus necessarily determines
the value of
$\max\{v_M(\tau):\tau\in -\beta_2y_1+M^{\wp}+K\}$.
However $-\beta_2y_1\in K(y_1)$, so we begin by working in the subfield $K(y_1)$, computing
\[\max\{v_{K(y_1)}(\tau):\tau\in -\beta_2y_1+K(y_1)^{\wp}+K\}.\]
This leads us to define, for a given ramified $C_p$-extension $L/K$,
the $L/K$-{\em group valuation} of an element $\ell\in L$:
\[\df_{L/K}(\ell)=\max\{v_L(x):x\in \ell+L^\wp+K\}.\]
Observe that  $\ell\in L^\wp+K$ if and only if $\df_{L/K}(\ell)=\infty$ and
\begin{equation}\label{val}
  \df_{L/K}(\ell+\ell')\geq \min\{\df_{L/K}(\ell),\df_{L/K}(\ell')\}
  \end{equation}
with equality when $\df_{L/K}(\ell)\neq\df_{L/K}(\ell')$.
Notice that if $\df_{L/K}(\ell)<\infty$ then $\df_{L/K}(\ell)<0$.
Given $\ell\in L$ with finite $L/K$-group valuation $\df_{L/K}(\ell)$, there exist $l\in L,k\in K$ such that
$v_L(\ell+\wp(l)+k)=\df_{L/K}(\ell)$. We will refer to this element
\[\URA{\ell}{L/K}=\ell+\wp(l)+k.\]
as a $L/K$-{\em reduction} of $\ell$.
Of course, while the valuation of the reduction $\URA{\ell}{L/K}$
is determined uniquely, the particular element $\URA{\ell}{L/K}$ that carries this valuation
is not. We will say that $\ell\in L$ is $L/K$-reduced if
$v_L(\ell)=v_L(\URA{\ell}{L/K})=\df_{L/K}(\ell)$.

The next result is a generalization of \eqref{wp}.
\begin{lemma}\label{image}
  Let $L/K$ be a ramified $C_p$-extension with ramification break $b$. Thus $L=K(y)$ for some $y\in K^\sep$ with $\wp(y)=\beta\in K$, $v_K(\beta)=-b$.
Let $\phi_{L/K}:[0,\infty)\rightarrow [0,\infty)$ be the Hasse-Herbrand function
      \[\phi_{L/K}(x)=\begin{cases}
      x &\mbox{ for }0\leq x\leq b,\\
      b+(x-b)/p&\mbox{ for }b< x\end{cases}\]
      with inverse $\psi_{L/K}$.
      Then for positive integers $n$ coprime to $p$,
      \[\wp:\frac{\caM_L^{-n}+K}{\caM_L^{-n+1}+K}\longrightarrow\frac{\caM_L^{-\psi_{L/K}(n)}+K}{\caM_L^{-\psi_{L/K}(n)+1}+K}\]
  is an isomorphism for $n\neq b$. If $n=b$ then $\ker\wp=\bF_py+y\caM_L+K$.
   \end{lemma}
\begin{proof}
    Since $p\nmid n$, there exist a unique pair $(i,m)$ with  $1\leq i\leq p-1$, $m\in \bZ$
    such that $n=bi+pm$. Every element of $(\caM_L^{-n}+K)/(\caM_L^{-n+1}+K)$ can be represented by $\mu y^i$ for some $\mu\in K$ with $v_K(\mu)=-m$.
    Since 
 $\wp(\mu y^i)=\mu^p(y+\beta)^i-\mu y^i=\mu^p\binom{i}{1}\beta^{i-1}y+(\mu^p-\mu)y^i\pmod{(\mu^p\beta^{i-1}y+\mu y)\caM_K+K}$, we find that 
  \[\wp: \frac{\caM_L^{pm-ib}+K}{\caM_L^{pm-ib+1}+K}\longrightarrow
  \begin{cases}
        \frac{\caM_L^{pm-ib}+K}{\caM_L^{pm-ib+1}+K}&\mbox{for }pm-ib\geq -b,\\
    \frac{\caM_K^{p^2m-(i-1)pb-b}+K}{\caM_K^{p^2m-(i-1)pb-b+1}+K}&\mbox{for }pm-ib< -b,
  \end{cases}\]
  from this the result follows. 
\end{proof}
  \begin{corollary}\label{>-b}
   If $\ell\in L$ satisfies $v_L(\ell)>-b$ then $\ell\in L^\wp+K$ and $\df_{L/K}(\ell)=\infty$.
\end{corollary}
\begin{corollary}\label{-b}
  If $\ell\in L$ satisfies $v_L(\ell)=-b$ then $\ell\in \omega y+y\caM_L$ for some $\omega\in \caO_K/\caM_K$. In this situation,\begin{center}
  $\ell\in L^\wp+K$ and $\df_{L/K}(\ell)=\infty$ if and only if $\omega\in (\caO_K/\caM_K)^\wp$.\end{center}
 \end{corollary}

 \begin{corollary}\label{complement}
   If $\ell\in L$ satisfies $v_L(\ell)<-b$ and either
   \[v_L(\ell)\not\equiv -b \bmod p\quad
   \mbox{ or }\quad
   v_L(\ell)\equiv (p-1)b\bmod p^2,\]
   then 
$\df_{L/K}(\ell)=   v_L(\ell)$.
 \end{corollary}

 \subsection{$L/K$-reductions of the terms in $\eus(x_1,x_2)$ for $\gD,\gH,\gM$}\label{subsect}

 Since $S(y_1,\beta_1)$ and $S(y_2,\beta_2)$ are associated with
 cyclic extensions of degree $p^2$ and ramification in cyclic
 extensions is well-understood, our focus in this section will be the
 $L/K$-reduction of $\pm\beta_2y_1$. Our first result decomposes
 $\beta_2$ into powers of $\beta_1$.
\begin{lemma}\label{decomp}
   Without loss of generality, the AS-generator $\beta_2$ satisfying $p\nmid v_K(\beta_2)\leq v_K(\beta_1)<0$, can be expressed as
\[
  \beta_2=\mu_0^p+\sum_{i=1}^{p-1}\mu_i^p\beta_1^i
\]
for some $\mu_i\in K$ such that
\begin{enumerate}[label=\alph*)]
   \item $\mu_0\in\caO_K/\caM_L$ and either $\mu_0\not\in\{ \wp(\omega):\omega\in \caO_K/\caM_K\}$ or $\mu_0=0$, and 
    \item for $1\leq i\leq p-1$, either $v_K(\mu_i^p\beta_1^i)<0$ or $\mu_i=0$.
       \end{enumerate}
Additionally, if $v_K(\beta_2)=v_K(\beta_1)$, we may suppose
$\mu_1\in\caO_M\setminus(\bF_p+\caM_K)$.
  \end{lemma}
\begin{proof}
  Observe that $K^p=\{\mu^p:\mu\in K\}$ is a subfield of $K$. Furthermore,
  $K=K^p(\beta_1)$ is a field extension of $K^p$ of degree $p$.
  Thus $1,\beta_1,\ldots, \beta_1^{p-1}$ is a basis for $K/K^p$ and there exists  $\mu_i\in K$ such that
  \[\beta_2=\sum_{i=0}^{p-1}\mu_i^p\beta_1^i.\]
However
  since $\beta_2$ is an AS-generator, we are only concerned with this statement as a congruence
  \begin{equation}\label{mu_0}
    \beta_2\equiv\mu_0^p+\sum_{i=1}^{p-1}\mu_i^p\beta_1^i\pmod{K^{\wp}}.\end{equation}
Consider the term $\mu_0^p$. If $v_K(\mu_0^p)<0$, then since
$\mu_0^p=\mu_0+\wp(\mu_0)$, we may express
$\mu_0=\sum_{i=0}^{p-1}(\mu_i')^p\beta_1^i$ for some $\mu'_i\in K$,
and find that
$\beta_2\equiv(\mu'_0)^p+\sum_{i=1}^{p-1}(\mu_i+\mu_i')^p\beta_1^i\pmod{K^{\wp}}$
where $v_K(\mu_0')>v_K(\mu_0)$. Repeat this process and relabel, until $v_K(\mu_0^p)\geq 0$.
Now if $\mu_0^p\equiv \wp(\omega)\pmod{\caM_K}$ for some $\omega\in \caO_K/\caM_K$, we set $\mu_0^p=0$. At this point, \eqref{mu_0} holds with $\mu_0\in \caO_K/\caM_K$
and either
$\mu_0\not\in\{ \wp(\omega):\omega\in \caO_K/\caM_K\}$ or $\mu_0=0$.
To finish up, we observe that
    since $\caM_K\subset K^\wp$, if
$v_K(\mu_i^p\beta_1^i)>0$ for some $1\leq i\leq p-1$, we may set $\mu_i=0$.
  \end{proof}

Our approach towards determining $\df_{K(y_1)/K}(\beta_2y_1)$
depends upon Lemma \ref{decomp} and \eqref{val}.
First we address the easy case when $p=2$.

\begin{proposition}\label{p=2}
   Assume $p=2$, and $y_1,y_2,\beta_1,\beta_2$ as above with
   $\beta_2$ as in Lemma \ref{decomp}.
   Then $\beta_2 y_1$ and $(\beta_1+\beta_2)y_1$ are $K(y_1)/K$-reduced with
  \[v_{K(y_1)}(\beta_2 y_1)=v_{K(y_1)}((\beta_1+\beta_2)y_1)=-2u_2-u_1.\]
  Similarly, $\beta_2 y_2$ is $K(y_2)/K$-reduced with $v_{K(y_2)}(\beta_2 y_2)=-3u_2$.
  \end{proposition}
\begin{proof}
  The results follow from Corollary \ref{complement} once it is observed that if
  $v_K(\beta_1)= v_K(\beta_2)$ then
  $\mu_1\not\in \bF_2+\caM_K$, and thus $v_K(\beta_1+\beta_2)=v_K(\beta_2)$.
 \end{proof}

Now we address the general case.
\begin{proposition}\label{p>2}
  Assume $p>2$, and $y_1,y_2,\beta_1,\beta_2$ as above with
   $\beta_2$ as in Lemma \ref{decomp}.
   Set
\[r=-v_K\left(\sum_{i=1}^{p-2}\mu_i^p\beta_1^i \right),\qquad s=-v_K(\mu_{p-1}^p\beta_1^{p-1}).\]
Observe
 $s\equiv -u_1\bmod p$, $r\not\equiv 0,-u_1 \bmod p$ and
$u_2=\max\{r,s\}$. Then
\[v_{K(y_1)}(\URA{\beta_2y_1}{K(y_1)/K})=\df_{K(y_1)/K}(\beta_2 y_1)=-\max\{ps+u_1,pr-(p-2)u_1\}.\]
  \end{proposition}
 \begin{proof}
   Let $L=K(y_1)$.
   Based upon Lemma \ref{decomp}, summands $\{\mu_i^p\beta_1^i\}_{i\neq 0}$
   in $\beta_2$ are either zero or have valuation $v_L(\mu_i^p\beta_1^i)< 0$.
   Decompose the sum
     $\sum_{i=1}^{p-2}\mu_i^p\beta_1^i=A+B$ where $A$ includes those summands
  satisfying $v_K(\mu_i^p\beta_1^i)\leq v_K(\beta_1)$ and $B$ the nonzero summands satisfying $v_K(\beta_1)< v_K(\mu_i^p\beta_1^i)<0 $.
  Thus
  \[\beta_2=A+\mu_{p-1}^p\beta_1^{p-1}+B+\mu_0^p.\]
  Since $v_K(\beta_2)\leq v_K(\beta_1)=-u_1$,
  \[v_K(\beta_2)=\min\{v_K(A),v_K(\mu_{p-1}^p\beta^{p-1})\}.\]
  At least one of $v_K(A),v_K(\mu_{p-1}^p\beta_1^{p-1})$ must be $\leq -u_1$.
  At least one of $A,\mu_{p-1}\in K$ is nonzero.

  To determine $\df_{L/K}(\beta_2y_1)$ we use 
  \eqref{val} and  consider the following $L/K$-group valuations:
\[\df_{L/K}(Ay_1),\quad \df_{L/K}(\mu_{p-1}^p\beta_1^{p-1}y_1), \quad\df_{L/K}(By_1)\quad\mbox{ and }\quad
\df_{L/K}(\mu_0^py_1).\] Two are easy to analyze.
\begin{itemize}
\item Consider $\df_{L/K}(\mu_{p-1}^p\beta_1^{p-1}y_1)$ and suppose $\mu_{p-1}\neq 0$. Then
    $v_L(\mu_{p-1}^p\beta_1^{p-1}y_1)\equiv
  (p-1)u_1 \bmod p^2$. Thus
  by
  Corollary
  \ref{complement}, $\mu_{p-1}^p\beta_1^{p-1}y_1$ is $L/K$-reduced and
  \[\df_{L/K}(\mu_{p-1}^p\beta_1^{p-1}y_1)=v_L(\mu_{p-1}^p\beta_1^{p-1}y_1)
    \leq -(p+1)u_1.\]
\item Consider $\df_{L/K}(\mu_0^py_1)$.
    Since either $\mu_0\not\in\{ \wp(\omega):\omega\in \caO_K/\caM_K\}$ or $\mu_0=0$, Corollary \ref{-b} states that
    \[-u_1\leq \df_{L/K}(\mu_{0}^py_1).\]
    \end{itemize}
The remaining two $L/K$-group valuations are more involved. However, once we prove that
if $A\neq 0$, then
\[\df_{L/K}(Ay_1)\equiv -2u_1\bmod p\quad\mbox{ and }\quad \df_{L/K}(Ay_1)\leq -2u_1,\]
while
$-2u_1< \df_{L/K}(By_1)$, we will be able to conclude that
\begin{equation}\label{goal}
  \df_{L/K}(\beta_2y_1)=\min\{\df_{L/K}(\mu_{p-1}^p\beta_1^{p-1}y_1),\df_{L/K}(Ay_1)\}.\end{equation}

We start by supposing that 
         $\mu\neq 0$ and $1\leq i\leq p-2$, and then expanding $\wp(\mu y_1^{i+1})=\mu^p(y_1+\beta_1)^{i+1}-\mu y_1^{i+1}$ to find that
  \[\wp(\mu y_1^{i+1})=\mu^p\beta_1^{i+1}+\mu^p(i+1)\beta_1^iy_1+
  \mu^p\sum_{j=2}^{i+1}\binom{i+1}{j}y_1^j\beta_1^{i+1-j}-\mu y_1^{i+1}.\]
  Since $2\leq i+1<p$, we may solve for $\mu^p\beta_1^iy_1$ finding that
  \[\mu^p\beta_1^iy_1\equiv \frac{i}{2}\mu^py_1^2\beta_1^{i-1}-\frac{\mu y_1^{i+1}}{i+1}\pmod{\mu^py_1^2\beta_1^{i-1}\caM_L+L^\wp+K}.\]
  Observe that
  \[v_K\left(\frac{i}{2}\mu^py_1^2\beta_1^{i-1}\right)<v_K\left(\frac{\mu y_1^{i+1}}{i+1}\right)\iff v_K(\mu^p\beta_1^i)< v_K(\beta_1).\]
  As a result, for $1\leq i\leq p-2$ and $\mu\neq 0$, while setting
  $\caL=L^\wp+K$ 
  \begin{equation}\label{relate}
    \mu^p\beta_1^iy_1\equiv\begin{cases} \frac{i}{2}\mu^py_1^2\beta_1^{i-1}\pmod{\mu^py_1^2\beta_1^{i-1}\caM_L+\caL} & v_K(\mu^p\beta_1^i)< v_K(\beta_1),\\
 \frac{\wp(\mu)}{2}y_1^2 \pmod{y_1^2\caM_L+\caL}& v_K(\mu^p\beta_1^i)= v_K(\beta_1),\\
 -\frac{\mu y_1^{i+1}}{i+1}\pmod{\mu y_1^{i+1}\caM_L+\caL} & v_K(\mu^p\beta_1^i)> v_K(\beta_1).\end{cases}\end{equation}

Now we apply
\eqref{relate}. Suppose that $A\neq 0$
and separate the cases: $v_K(A)<v_K(\beta_1)$ vs.~$v_K(A)=v_K(\beta_1)$.
In the first case $v_K(A)<v_K(\beta_1)$, let
  $\mu_i^p\beta_1^iy_1$ be any
nonzero summand of
  $Ay_1$ such that $v_L(\mu_i^p\beta_1^i)<v_L(\beta_1)$,
then by  \eqref{relate} it
  is congruent modulo $L^\wp+K$ to a term of valuation
  $v_L(\mu_i^p\beta_1^{i-1}y_1^2)<-2u_1$. Moreover, using Corollary
  \ref{complement} $\mu_i^p\beta_1^{i-1}y_1^2$ has largest valuation in the
  coset of $L^\wp+K$ that it represents. Thus 
  \[\df_{L/K}(\mu_i^p\beta_1^iy_1)=v_L(\mu_i^p\beta_1^i)+(p-2)u_1<-2u_1.\]
  If there is a summand such that $v_L(\mu_i^p\beta_1^i)=v_L(\beta_1)$, then $i=1$ and $v_K(\mu_i)=0$ and
  $\mu_i^p\beta_1^iy_1$ is
  congruent modulo $L^\wp+K$ to a term of valuation
  $-2u_1\leq v_L(\wp(\mu_i)y_1^2)$. Using \eqref{val}, we conclude that
  \[\df_{L/K}(Ay_1)=v_L(A)+(p-2)u_1<-2u_1.\]

  The second case $v_K(A)=v_K(\beta_1)$ occurs when $A$ has only one
  summand $\mu_i^p\beta_1^i$ with $i=1$ and $v_K(\mu_i)=0$. This means that
  $v_K(A)=v_K(\beta_2)$ and
  \[\beta_2\equiv \mu_1^p\beta_1\bmod \beta_1\caM_K.\]
  Since $u_1=-v_K(\beta_1)=-v_K(\beta_2)=u_2$, there is only one ramification break in the $C_p^2$-extension $M/K$, every nontrivial $\bF_p$-linear combination of $\beta_1$ and $\beta_2$ has the same valuation, $\mu_1\in\caO_K\setminus(\bF_p+\caM_K)$, and thus $v_K(\wp(\mu_1))=0$. In this  case,
  \[\df_{L/K}(Ay_1)=v_L(A)+(p-2)u_1=-2u_1.\]

  Finally, we apply \eqref{relate} to $By_1$.
  The condition $v_K(\mu^p\beta_1^i)> v_K(\beta_1)$ is equivalent to
  $v_L(\mu y_1^i)> v_K(y_1)=-u_1$. Thus
  each summand
  $\mu_i^p\beta_1^iy_1$ of $By_1$, is congruent modulo $L^\wp+K$ to a
  term of valuation $v_L(-\mu y_1^{i+1})$, which satisfies $-2u_1<v_L(-\mu y_1^{i+1})$. So
    \[-2u_1< \df_{L/K}(By_1).\]
    The result now follows from \eqref{goal} and $\df_{L/K}(Ay_1)=v_L(A)+(p-2)u_1$.
\end{proof}

 We now record results that involve $S(y_1,\beta_1)$ or $S(y_2,\beta_2)$.
 
 \begin{proposition}\label{S(x_2)}
   Assume $p>2$, and $y_1,y_2,\beta_1,\beta_2$ as above with
   $\beta_2$ as in Lemma \ref{decomp}.  Then
   $S(y_2,\beta_2)$ is $K(y_2)/K$-reduced with
   $v_L(S(y_2,\beta_2))=-(p^2-p+1)u_2$.  Set $r,s$ as in Proposition
   \ref{p>2}. Then
\[\df_{K(y_1)/K}(-\beta_2 y_1+S(y_1,\beta_1))=-\max\{(p^2-p+1)u_1,ps+u_1,pr-(p-2)u_1\},\]
except when $\mu_{p-1}+1\in  \caM_K$.

When $\mu_{p-1}+1\in \caM_K$,
 $s=(p-1)u_1$ is fixed. Set
\[t=-v_K((\mu_{p-1}+1)^p\beta^{p-1})<s.\]
Note that $t\equiv -u_1\bmod p$.
Then
\[\df_{K(y_1)/K}(-\beta_2 y_1+S(y_1,\beta_1))=-\max\{(p^2-2p+2)u_1,pt+u_1,pr-(p-2)u_1\}.\]
  \end{proposition}
 \begin{proof}
   Since $v_{K(y_i)}(S(y_i,\beta_i))=-(p^2-p+1)u_i$, we conclude from
   Corollary \ref{complement} that $\df_{K(y_i)/K}(S(y_i,\beta_i))=-(p^2-p+1)u_i$.

   Now consider
   $-\beta_2y_1+S(y_1,\beta_1)=C-D$
 where
 \[C=-\mu_{p-1}^p\beta_1^{p-1}y_1+S(y_1,\beta_1)=-(\mu_{p-1}^p+1)\beta_1^{p-1}y_1-\sum_{i=1}^{p-2}\frac{1}{p}\binom{p}{i}\beta_1^iy_1^{p-i}\]
 and $D=Ay_1+By_1+\mu_0^py_1$ is known from Proposition
 \ref{p>2} to satisfy $\df_{K(y_1)/K}(D)=-(pr-(p-2)u_1)$.
 Suppose $\mu_{p-1}+1\not\in\caM_K$. Then
based upon Corollary \ref{complement},
\[\df_{K(y_1)/K}(C)=\min\{v_{K(y_1)}(-\mu_{p-1}^p\beta_1^{p-1}y_1),v_{K(y_1)}(S(y_1,\beta_1))\}\equiv (p-1)u_1\bmod p^2,\]
and
thus
$\df_{K(y_1)/K}(C)\neq
\df_{K(y_1)/K}(D)$. The first statement follows.
On the other hand, if $\mu_{p-1}+1=\tau\in\caM_K$ then
$v_L(C)=\min\{v_L(\tau^p\beta_1^{p-1}y_1), v_L(\beta_1^{p-2}y_1^2)\}\equiv (p-1)u_1$ or $(2p-2)u_1\bmod p^2$. Observe that
$-(pr-(p-2)u_1)\equiv -2u_1\bmod p$ and
since $r\not\equiv (p-1)u_1\bmod p$, 
$-(pr-(p-2)u_1)\not\equiv (2p-2)u_1\bmod p^2$.
Again
$\df_{K(y_1)/K}(C)\neq
\df_{K(y_1)/K}(D)$. The second statement follows.
 \end{proof}
 \begin{remark}\label{=same-as>}
   When $p=2$ the sum that produces $r$ is an empty sum. Thus $s=u_2$
   and the value of $\df_{K(y_1)/K}(\beta_2y_1)$ given in Proposition
   \ref{p>2} agrees with the value in Proposition \ref{p=2}.  The same
   can be said for the values of
   $\df_{K(y_1)/K}(-\beta_2y_1+S(y_1,\beta_1))$ and
   $\df_{K(y_1)/K}(S(y_2,\beta_2))$ given in Proposition \ref{S(x_2)}.
   They also agree with the values in Proposition \ref{p=2}.
 \end{remark}

  \subsection{Decomposition and reduction of $\eus(x_1,x_2)$ for $\gQ$}\label{Q-subsect}
 In this case, $\chr(K)=2$ and since $x_i=y_i$ for $i=1,2$,
 we also have $u_i=-v_K(\kappa_i)$. We will continue to use $x_i$ and $\kappa_i$
 (rather than $y_i$ and $\beta_i$). Recall that 
$N=M(x_3)$ where $\wp(x_3)=(\kappa_1+\kappa_2)x_1+\kappa_2x_2+\kappa_3$ for some $\kappa_3$. In fact, as $\kappa_3$ ranges over all of $K$, $M(x_3)$ ranges over all totally ramified quaternion extensions of $K$. We are interested in determining a lower bound on the ramification break of $M(x_3)/M$. Thus we are interested in 
$\max\{v_M(t):t\in \eus'+M^{\wp}+K\}$ for
\[\eus'=\eus(x_1,x_2)=(\kappa_1+\kappa_2)x_1+\kappa_2x_2.\]

Using Lemma \ref{decomp}, we have
\begin{equation}\label{kappa2}
\kappa_2=\mu^2\kappa_1+\mu_0^2
\end{equation}
where $-v_K(\mu)=m\geq 0$ and $\mu_0\in \caO_K/\caM_K$. Notice that
\[u_2=u_1+2m.\]
If $m=0$ then since
$\kappa_1,\kappa_2$ are linearly independent in
$\kappa_1\caO_K/\kappa_1\caM_K$, we have $\wp(\mu)\neq 0$, which means that $\wp(\mu)$ is a unit
in $\caO_K/\caM_K$.
Set \begin{equation}\label{X}
  X=x_2-\mu x_1
\end{equation}
  and observe that
$\wp(X)=-\wp(\mu) x_1+\mu_0^2$.
This means that $X$ satisfies an Artin-Schreier polynomial over $L=K(x_1)$.
Since $v_L(-\wp(\mu) x_1+\mu_0^2)=pv_K(\wp(\mu))-u_1\not\equiv 0\bmod 2$, we
see that $L(X)/L$ is a $C_2$-extension with ramification break
$b=u_1+4m$ and
\[v_M(X)=-(u_1+4m).\]

Using \eqref{kappa2} and \eqref{X},
replace $\kappa_2$ and $x_2$ in $\eus'$ so that
$\eus'=\eus_1'+\eus_2'$ with
\begin{align*}
  \eus_1'&=(1+\mu^2+\mu^3)\kappa_1 x_1+\mu_0^2(1+\mu) x_1\in L=K(x_1),\\
  \eus_2'&=(\mu^2\kappa_1+\mu_0^2)X\in M=K(x_1,x_2).
\end{align*}
Observe that
\[v_M(\eus_2')=-(5u_1+12m), \quad\mbox{and if $m>0$ then }
v_L(\eus_1')=-(3u_1+6m),\]
but that the determination of $v_L(\eus_1')$ is not so clear when $m=0$.

To clarify matters when $m=0$, we
replace $\mu^3\kappa_1x_1$ in the expression for $\eus_1'$, by expanding
$\wp(\mu x_1X+\mu X)$ to find that
\[\mu^3\kappa_1x_1\equiv \mu^4\kappa_1x_1+\mu^2\mu_0^2x_1+\mu^2\kappa_1X+\wp(\mu)x_1X+\wp(\mu)X\pmod{M^\wp+K}.\]
Thus $\eus'\equiv \eus\pmod{M^\wp+K}$ 
where $\eus=\eus_1+\eus_2$ with
\begin{align*}
  \eus_1&=(1+\mu^2+\mu^4)\kappa_1 x_1+\mu_0^2(1+\mu+\mu^2) x_1\in L=K(x_1),\\
  \eus_2&=\wp(\mu)x_1X+\wp(\mu)X+\mu_0^2X\in M=K(x_1,x_2).
\end{align*}
When $m=0$, express $\mu=\omega+\epsilon$ for some $\omega\in (\caO_K/\caM_K)\setminus\bF_2$ and some $\epsilon\in\caM_K$ with $v_K(\epsilon)=e>0$.
Since $\wp(\omega)$ is a unit, $v_K(\eus_2)=-3u_1$.
Observe that if $e\geq u_1/2$ then $(\omega+\epsilon)^2\kappa_1\equiv \omega^2\kappa_1\bmod \caM_K\subseteq K^\wp$ and thus, since we are only interested in $\kappa_2=\mu^2\kappa+\mu_0^2\bmod K^\wp$, we may set $\epsilon=0$. Without loss of generality, we conclude that either $0<e<u_1/2$ or $\epsilon=0$. Replace $\mu$ in the expression for $\eus_1$:
\[
  \eus_1=(1+\omega+\omega^2)^2\kappa_1x_1+(1+\omega+\omega^2)\mu_0^2x_1+(\epsilon+\epsilon^2)^2\kappa_1x_1+(\epsilon+\epsilon^2)\mu_0^2x_1.
  \]
  It is clear that $v_L(\eus_1)$ depends upon whether $\omega^3\neq 1$
  or $\omega^3=1$.  If $\omega^3\neq 1$, then $v_L(\eus_1)= -3u_1$ as
  $m=0$. If $\omega^3=1$, then
  \[
  \eus_1=(\epsilon+\epsilon^2)^2\kappa_1x_1+(\epsilon+\epsilon^2)\mu_0^2x_1
  \]
  and $v_L(\eus_1)=-3u_1+4e$. Note that $\eus_1=0$ if  $\omega^3=1$ and $v_K(\epsilon)=e>u_1/2$.
Altogether, this means that when $m=0$, 
\[v_M(\eus_2)=-3u_1, \quad\mbox{and}\quad
v_L(\eus_1)=\begin{cases}
-3u_1 &\omega^3\neq 1,\\
-3u_1+4e&\omega^3= 1\mbox{ and } \eus_1\neq0.\end{cases}\]

We consolidate this information in a proposition.
\begin{proposition}\label{Q8-s1-s2}
  Let $p=2$. Let $M/K$ be a totally ramified $C_p^2$ extension with upper ramification numbers $u_1\leq u_2$. Set $M=K(x_1,x_2)$
  where $\wp(x_i)=\kappa_i$ with
  $v_k(\kappa_i)=u_i$. Since $u_1,u_2$ are odd, $\kappa_2\equiv\mu^2\kappa_1\bmod\caO_M$ for some $\mu\in K$ with $v_K(\mu)=-m\leq 0$. If $m=0$ then
  $\mu=\omega+\epsilon$ for some $\omega\in \caO_K/\caM_K$ and $\epsilon\in \caM_K$ with $v_K(\epsilon)=e$. Let $L=K(x_1)$. Then there exist $\eus_1\in L$ and $\eus_2\in M$ such that for $u_2\neq u_1$ or equivalently $m>0$,
  \[v_L(\eus_1)=-3u_2\mbox{ and }v_M(\eus_2)=-(6u_2-u_1).\]
For $u_2=u_1$ or equivalently $m=0$, we have $v_M(\eus_2)=-3u_1$. And
unless $\eus_1=0$, we have
  \[v_L(\eus_1)=\begin{cases}-3u_1 &\mbox{ if }\omega^3\neq 1,\\
  -3u_1+4e &\mbox{ if $\omega^3= 1$ and $0<e<u_1/2$}.
  \end{cases}\]
    Note that for $\eus_1\neq 0$, the inequality
  $v_L(\eus_1)<-u_1$ holds.
  Finally,
  every $\gQ$-extension $N/K$ that contains $M$ is expressible as $N=M(x_3)$ where $\wp(x_3)=\eus_1+\eus_2+\kappa_3$ for some $\kappa_3\in K$.
\end{proposition}

\section{Ramification theory}\label{linking}
Recall the notation thus far: $N/K$ is a totally ramified nonabelian
extension with Galois group $G=\Gal(N/K)$ generated by
$\sigma_1,\sigma_2$ where $\sigma_3=[\sigma_1,\sigma_2]$ fixes a
subfield $M$.  The fixed field $M=N^{\sigma_3}$ was initially
expressed as $M=K(x_1,x_2)$ when we were solely concerned with Galois
action. Now that we need the ramification breaks to be involved, we
set $M=K(y_1,y_2)$
such that $\wp(y_i)=\beta_i$ and $u_i=-v_K(\beta_i)$ for $i=1,2$.
Recall that we proved that if
$u_1\leq u_2\leq u_3$ denote the upper and
$l_1\leq l_2\leq l_3$ the lower ramification breaks of $N/K$, then
the lower and upper ramification breaks of $M/K$ are
$l_1\leq l_2$ and $u_1\leq u_2$, respectively.
To determine $l_3$,
ramification break of $N/M$, so that $u_3$ is determined by
$l_3-l_2=p^2(u_3-u_2)$, we are examine 
the ramification break of the
$C_p$-extension $M(x_3)/M$ where \[\wp(x_3)=\eus(x_1,x_2)+\kappa_3\]
with $\eus(x_1,x_2)$ first described in Theorem \ref{ram-as-gen} then translated/adjusted into \eqref{s(y,y)} and 
\eqref{gD,gM-adjust}. Finally,
using the results of
  \S\ref{subsect} and \S\ref{Q-subsect}, namely Propositions \ref{p=2},
  \ref{p>2}, \ref{S(x_2)} and \ref{Q8-s1-s2}, we find that $\eus(x_1,x_2)$ can be replaced modulo $M^\wp+K$ by
  \[
  \eus(x_1,x_2)\equiv\begin{cases}
-\URA{\beta_2y_1}{L/K}& \gH,\\
\URA{-\beta_2y_1+S(y_1,\beta_1)}{L/K}& \gD,\gM, (u_2=u_1\mbox{ or }G_{l_2}=\langle \sigma_1^p,\sigma_2\rangle),\\
\URA{\beta_2y_1}{L/K}+S(y_2,\beta_2)& \gD,\gM, u_2\neq u_1\mbox{ and }G_{l_2}=\langle \sigma_1\rangle,\\
\eus_1+\eus_2& \gQ.
  \end{cases}\]
This leads naturally to an interest in the ramification breaks of the following auxiliary $C_p$-extensions:
\begin{itemize}
\item $L(z_0)/L$ for $z_0\in K^\sep$ such that $\wp(z_0)=\URA{\beta_2y_1}{L/K}$,
\item $L(z_1)/L$ for $z_1\in K^\sep$ such that $\wp(z_1)=\URA{-\beta_2y_1+S(y_1,\beta_1)}{L/K}$,
\item $K(y_2,z_2)/K(y_2)$ for $z_2\in K^\sep$ such that $\wp(z_2)=S(y_2,\beta_2)$,
\item $L(z_3)/L$ for $z_4\in K^\sep$ such that $\wp(z_3)=\eus_1$ if $\eus_1\neq 0$,
\item $M(z_4)/M$ for $z_4\in K^\sep$ such that $\wp(z_4)=\eus_2$,
\end{itemize}
which we attach
to a diagram of $M/K$, and where
for easy reference, we  label each $C_p$-subextension with its ramification break. The purpose of this diagram is to help us determine the ramification break of $M(x_3)/M$ when $x_3$ as in \eqref{x_3}. Thus we add the extension $M(x_3)/M$ to our diagram to remind us of the ``target'' in this exercise.
\begin{center}
    \begin{tikzpicture}

    \node (Q1) at (0,0) {$\mathbf{K}$};
    \node (Q2) at (2,2) {$\mathbf{K(y_2)}$};
    \node (Q3) at (0,4) {\hspace*{2cm}$\mathbf{M=K(y_1,y_2)}$};
    \node (Q4) at (-2,2) {$\mathbf{L=K(y_1)}$};
        \node (Q6) at (0,6) {$N=M(x_3)$};
    \node (Q7) at (-3,4) {$L(z_1)$};
    \node (Q8) at (-4.5,4) {$L(z_0)$};
    \node (Q9) at (4,4) {$K(y_2,z_2)$};
    \node (Q10) at (-1.8,4) {$L(z_3)$};
    \node (Q11) at (2.25,6) {$M(z_4)$};

    \draw (Q1)--(Q2) node [pos=0.5, above,inner sep=0.25cm] {$\scriptstyle{u_2}$};
    \draw (Q1)--(Q4) node [pos=0.7, below,inner sep=0.25cm] {$\scriptstyle{u_1}$};
    \draw (Q4)--(Q3) node [pos=0.7, below,inner sep=0.25cm] {$\scriptstyle{l_2}$};
    \draw (Q2)--(Q3) node [pos=0.7, below,inner sep=0.25cm] {$\scriptstyle{u_1}$};
        \draw[dotted] (Q3)--(Q6) node [pos=0.6, right,inner sep=0.15cm] {$\scriptstyle{l_3}$};
    \draw (Q4)--(Q7) node [pos=0.35, above,inner sep=0.35cm] {$\scriptstyle{t_1}$};
    \draw (Q4)--(Q8) node [pos=0.5, above,inner sep=0.15cm] {$\scriptstyle{t_0}$};
    \draw (Q2)--(Q9) node [pos=0.7, below,inner sep=0.25cm] {$\scriptstyle{t_2}$};
    \draw (Q4)--(Q10) node [pos=2.6, below,inner sep=2.65cm] {$\scriptstyle{t_3}$};
    \draw (Q3)--(Q11) node [pos=0.7, below,inner sep=0.25cm] {$\scriptstyle{t_4}$};

    \end{tikzpicture}
\end{center}
Using
Propositions \ref{p=2},
  \ref{p>2} and \ref{S(x_2)},
   the ramification breaks are determine and recorded below. Note that by
  Remark \ref{=same-as>}, these expressions for the ramification breaks hold for $p=2$ as well as for $p>2$.
\begin{align*}
  t_0&=\max\{ps+u_1,pr-(p-2)u_1\},\\
  t_1&=\max\{(p^2-p+1)u_1,ps+u_1,pr-(p-2)u_1\},\mbox{ or}\\
  &\max\{(p^2-2p+2)u_1,pt+u_1,pr-(p-2)u_1\}\mbox{ and }s=(p-1)u_1,\\
  t_2&=(p^2-p+1)u_2,\\
  t_3&=3u_2\mbox{ if }u_2\neq u_1,\mbox{ otherwise }t_3>u_1,\\
  t_4&=6u_2-u_1 \mbox{ if }u_2\neq u_1,\mbox{ otherwise }t_3=3u_1.\\
  \end{align*}
    Because the elements $\URA{\beta_2y_1}{L/K}$,
    $\URA{-\beta_2y_1+S(y_1,\beta_1)}{L/K}$, $\eus_1$ and
    $S(y_2,\beta_2)$ all lie within proper subfields of $M$, while we
    are interested in the $M$-group valuation of these elements, we record the following:
    \begin{lemma}
The inequalities $t_0,t_1,t_3>l_2$ and $t_2>u_1$ hold.
    \end{lemma}
    \begin{proof}
      Prove $t_0>l_2=pu_2-(p-1)u_2$ by checking the cases: $u_2=r$ and $u_2=s$. Next we prove $t_1>l_2$, which is clear when $t_1=\max\{(p^2-p+1)u_1,ps+u_1,pr-(p-2)u_1\}$. So consider the case when
      $t_1=\max\{(p^2-2p+2)u_1,pt+u_1,pr-(p-2)u_1\}$ and $s=(p-1)u_1$.
      If $u_2=r$, then $t_1>l_2$ follows as before.
      On the other hand, if $u_2=s=(p-1)u_1$, then $l_2=p(p-1)u_1-(p-1)u_1<(p^2-2p+2)u_1$, which also gives $t_1>l_2$.
      Prove $t_3>l_2$ by checking the cases $u_2=u_1$ and $u_2>u_1$.
      Finally, $t_2>u_1$ because $t_2=(p^2-p+1)u_2\geq 3u_2$.
      \end{proof}
Now that we have established these inequalities,
    we need a well-known
    result.
\begin{lemma}\label{C_p^2-breaks}
  Let $M/L$ be a ramified $C_p$-extension with the ramification break $l$. Let $\alpha\in L$ be $L$-reduced: $\df_{L}(\alpha)=v_L(\alpha)=-a<0$ and $p\nmid a$. Suppose
  \[a>l.\] Let $z\in L^\sep$ such that
  $\wp(z)=\alpha$.  Then $M(z)/M$ is a ramified $C_p$-extension with
  ramification break $pa-(p-1)l$ and \[\df_{M}(\alpha)=-(pa-(p-1)l).\]
  Otherwise if $a\leq l$ and 
  $M(z)/M$ is nontrivial, the ramification
  break of $M(z)/M$ is less than or equal to $l$.
 Sharper upper bounds than this exist, but this is enough for our purpose.
\end{lemma}
\begin{proof} Suppose $l<a$.
    Since $M(z)/L$ is a totally ramified $C_p^2$-extension with upper
    ramification breaks $l<a$, the lower ramification breaks are
    $l_1=l$ and $l_2=l+p(a-l)$.  Passing to the ramification
    filtration of $\Gal(M(z)/M)$ yields the result. More generally, when
    $M(z)/M$ is nontrivial, the ramification break of
    $M(z)/M$ satisfies $\max\{a,pa-(p-1)l\}$ with equality when $a\neq l$.
  \end{proof}

We may express $x_3=\bar{x}_3+z_5$ where $\wp(z_5)=\kappa_3$ and
\begin{equation}\label{x_3}
\bar{x}_3=\begin{cases}
 -z_0  &\gH,\\
z_1& \gD,\gM, ( u_2=u_1\mbox{ or }G_{l_2}=\langle \sigma_1^p,\sigma_2\rangle),\\
z_0+z_2
& \gD,\gM,u_2\neq u_1\mbox{ and }G_{l_2}=\langle \sigma_1\rangle,\\
z_3+z_4& \gQ.
\end{cases}
\end{equation}
Letting $s_i$ denote the ramification break for $M(z_i)/M$ we find,
based upon Lemma \ref{C_p^2-breaks}, that $s_i=pt_i-(p-1)l_2$ for
$i=0,1,3$, $s_2=pt_2-(p-1)u_1$ and of course, $s_4=t_4$.

At this point, we have collected all the information we need to
determine the ramification break $\bar{l}_3$ of $M(\bar{x}_3)/M$ where $\bar{x}_3$ is
expressed as in \eqref{x_3}.
Consider the cases when $G\cong \gH$,
or $G\cong \gD,\gM$ and $G_{l_2}=\langle \sigma_1^p,\sigma_2\rangle$.
Thus $\bar{l}_3=s_i$ for $i=0,1$. The upper ramification numbers for $M(\bar{x}_3)/K$ are $u_1,u_2,\bar{u}_3$ where
$\bar{u}_3-u_2=(\bar{l}_3-l_2)/p^2=(t_i-l_2)/p$. Separating $\gD$ off from $\gM$ for clarity, this establishes the fact that
\begin{equation}\label{bar{u}-1}
\bar{u}_3=\begin{cases}
\max\left\{s+u_1,r+\frac{u_1}{p}\right\} & \gH,\\
u_2+u_1 & \gD, (u_2=u_1\mbox{ or }G_{l_2}=\langle \sigma_1^p,\sigma_2\rangle),\\
\max\left\{pu_1,s+u_1,r+\frac{u_1}{p}\right\} &\gM, (u_2=u_1\mbox{ or }G_{l_2}=\langle \sigma_1^p,\sigma_2\rangle),\\
&\hspace*{1.5cm}\mbox{ and }\mu_{p-1}\neq 1\bmod \caM_K\\
\max\left\{(p-1)u_1+\frac{u_1}{p},t+u_1,r+\frac{u_1}{p}\right\} &
\gM, (u_2=u_1\mbox{ or }G_{l_2}=\langle \sigma_1^p,\sigma_2\rangle),\\
&\hspace*{1.5cm}\mbox{ and } \mu_{p-1}= 1\bmod \caM_K.\\
\end{cases}\end{equation}
Now consider the case where
 $G\cong \gQ$, or
$G\cong \gD,\gM$ and $u_2\neq u_1$, $G_{l_2}=\langle \sigma_1\rangle$. In these cases, $\bar{x}_3$ is the sum of two terms. Indeed, is an element of $M(z_0,z_2)$ or $M(z_4,z_5)$.
\begin{lemma}
  $M(z_0,z_2)$ is a $C_p^2$-extension with two upper ramification breaks $s_0<s_2$. If $u_2\neq u_1$, 
  then $M(z_3,z_4)$ is $C_p^2$-extension with
  upper ramification breaks then $s_3<s_4$.
  Otherwise, if $\eus_1\neq 0$, the upper ramification breaks are
  $s_4 <s_3$.
\end{lemma}
\begin{proof}
  Verify that $s_2>s_0$ follows from the fact that $pu_2+u_1\geq t_0$ and $u_2>u_1$. For $u_2\neq u_1$, $s_3<s_4$ reduces to $u_1<u_2$. For
  $u_2= u_1$
and $\eus_1\neq 0$, $s_4<s_3$ reduces to $e<u_1/2$.
\end{proof}

We use this lemma for the case when $G\cong \gD,\gM$, $u_2\neq u_1$
and $G_{l_2}=\langle \sigma_1\rangle$, and also when $G\cong \gQ$.
The ramification break of $M(\bar{x}_3)/M$ where $\bar{x}_3=z_0+z_2$
is $s_2$, which is also the third lower ramification break for
$M(\bar{x}_3)/K$.  The ramification break of $M(\bar{x}_3)/M$ where
$\bar{x}_3=z_3+z_4$ is $s_4$ if $u_2\neq u_1$. On the other hand, when
$u_2=u_1$ and $\eus_1\neq 0$, it is $s_3$. And if $u_2=u_1$ and
$\eus_1= 0$ (so $M(z_3,z_4)$ is a $C_p$-extension), it is $s_4$. These
are also the third lower ramification breaks for $M(\bar{x}_3)/K$.
Thus we determine that:
\begin{equation}\label{bar{u}-2}
\bar{u}_3=\begin{cases}
pu_2&\gD,\gM, u_2\neq u_1\mbox{ and }G_{l_2}=\langle \sigma_1\rangle,\\
2u_2 & \gQ, u_2\neq u_1,\mbox{ or }u_2= u_1,\omega^3\neq 1,\\
\max\left\{\frac{3u_1}{2},2u_1-2e\right\} & \gQ, u_2=u_1, \omega^3= 1.
\end{cases}\end{equation}

It is an easy exercise to determine from \eqref{bar{u}-1} and
\eqref{bar{u}-2} that either $\bar{u}_3$ is not an integer (because
$p\nmid u_1$) or $\bar{u}_3$ is an integer congruent to zero modulo
$p$.  Now recall that regardless of the Galois group, $N=M(x_3)$ with
$x_3=\bar{x}_3+z_5$ where $\wp(z_5)=\kappa_3$. Without loss of
generality, $\kappa_3$ is $K$-reduced. This means that unless
$\kappa_3\in\euO_K$, $v_K(\kappa_3)=-b_3<0$ with $p\nmid b_3$, in
which case $b_3$ is the ramification break of $K(z_5)/K$ and an upper
ramification break for the Galois extension $M(\bar{x}_3,z_5)/K$,
which contains $M(x_3)/K$. The upper ramification breaks of
$M(\bar{x}_3,z_5)/K$ include $u_1\leq u_2<\bar{u}_3$ as well.  And, since
$\bar{u}_3$ and $b_3$ are different types of rational numbers (never equal),
the upper ramification breaks of $M(x_3)/K$ are $u_1\leq u_2<\max\{\bar{u}_3,b_3\}$. This together with the expressions in \eqref{bar{u}-1} and
\eqref{bar{u}-2} appear in 
Theorem \ref{sharp-bound}.

\section{Application}\label{notting}

As explained in \cite{nottingham}, there is interest in explicit
constructions of finite nonabelian subgroups of the Nottingham group,
The authors describe a process that uses the Witt vector description
of cyclic extensions of degree $p^n$ in characteristic $p$ to produce
elements of order $p^n$. It would be interesting to follow this
process and use the Artin-Schreier descriptions in Theorem
\ref{as-gen} to identify in the Nottingham group some nonabelian subgroups of
order $p^3$. Furthermore, one might then use Theorem \ref{sharp-bound}
to determine the upper ramification sequences for these subgroups and
thus address one of the open problems listed in
\cite[\S1.5]{nottingham}.

\bibliographystyle{amsalpha}

\bibliography{bib}

\end{document}